\documentclass{amsart}

\usepackage{amsmath}
\usepackage{amssymb}
\usepackage{amscd}
\usepackage{hyperref}

\newtheorem{theorem}{Theorem}[section]
\newtheorem{thm}[theorem]{Theorem}

\newtheorem{lem}[theorem]{Lemma}

\theoremstyle{definition}

\newtheorem{constr}[theorem]{Construction}

\theoremstyle{remark}

\newcommand{\mbb}{\mathbb}

\newcommand{\NN}{\mbb{N}}
\newcommand{\ZZ}{\mbb{Z}}
\newcommand{\CC}{\mbb{C}}

\newcommand{\PP}{\mbb{P}}

\newcommand{\mc}{\mathcal}

\newcommand{\mcX}{\mc{X}}

\newcommand{\OO}{\mc{O}}

\newcommand{\wht}{\widehat}
\newcommand{\whts}{\wht{s}}

\newcommand{\whtOO}{\wht{\OO}}

\newcommand{\SP}{\text{Spec }}

\newsavebox{\sembox}
\newlength{\semwidth}
\newlength{\boxwidth}

\newsavebox{\semrbox}
\newlength{\semrwidth}
\newlength{\boxrwidth}

\newcommand{\Semr}[1]{%
\sbox{\semrbox}{\ensuremath{#1}}%
\settowidth{\semrwidth}{\usebox{\semrbox}}%
\sbox{\semrbox}{\ensuremath{\left(\usebox{\semrbox}\right)}}%
\settowidth{\boxrwidth}{\usebox{\semrbox}}%
\addtolength{\boxrwidth}{-\semrwidth}%
\left(\hspace{-0.3\boxrwidth}%
\usebox{\semrbox}%
\hspace{-0.3\boxrwidth}\right)%
}

\title
{$R$-equivalence on del Pezzo surfaces of degree $4$ and cubic surfaces}

\author[Tian]{Zhiyu Tian}
\address{
Department of Mathematics 253-37\\
California Institute of Technology \\
Pasadena, CA, 91125}
\email{tian@caltech.edu}

\date{\today}

\begin{document}


\begin{abstract}
We prove that there is a unique $R$-equivalence class on every del Pezzo surface of degree $4$ defined over the Laurent field $K=k\Semr{t}$ in one variable over an algebraically closed field $k$ of characteristic not equal to $2$ or $5$. We also prove that given a smooth cubic surface defined over $\CC\Semr{t}$, if the induced morphism to the GIT compactification of smooth cubic surfaces lies in the stable locus (possibly after a base change), then there is a unique $R$-equivalence class.
\end{abstract}


\maketitle


\section{Introduction}

Given a variety $X$ over a field $K$, Manin \cite{ManinCubic} defined the $R$-equivalence relation on the set of rational points $X(K)$. Recall that two points $x, y$ are \emph{directly R-connected} if there is a $K$-morphism $f: \PP^1_K \to X$ such that $f(0)=x, f(\infty)=y$. Then the $R$-equivalence relation is the equivalence relation generated by such relations.

The $R$-equivalence relation on $X(K)$ is not so interesting unless $X$ already contains lots of rational curves (at least over an algebraic closure of $K$). Such $X$ are rationally connected. For a precise definition of rational connectedness, see \cite{Kollar96}.

Many people have studied the $R$-equivalence classes on cubic hypersurfaces \cite{MadoreCubic}, intersection of two quadrics \cite{CTSK}, \cite{CTSD}, and low degree complete intersections \cite{PirutkaRequiv}.

In this short note, we give a very simple proof of the following result.
\begin{thm}\label{thm:degree4}
Let $X$ be a smooth del Pezzo surface of degree $4$ defined over the Laurent field $K=k\Semr{t}$ in one variable over an algebraically closed field $k$ of characteristic not equal to $2$ or $5$. Then there is exactly one $R$-equivalence class on $X(K)$.
\end{thm}

For the $R$-equivalence classes on a del Pezzo surface of degree $4$, J.-L. Colliot-Th{\'e}l{\`e}ne and A.~N. Skorobogatov proved that there is only one $R$-equivalence class if the surface has a conic bundle structure with $4$ degenerate fibers \cite{CTSK} over a field of cohomological dimension $1$. Very recently, Colliot-Th{\'e}l{\`e}ne proved that there is only one $R$-equivalence class if the field is $C_1$ and characteristic $0$ \cite{CTdelPezzo4}.

The observation in this paper in the case of Laurent field one can use a simple geometric argument to prove the statement. In the same spirit and using the $G$-equivariant techniques developed in \cite{WAIsotrivial}, we prove the following.
\begin{thm}\label{thm:cubic}
Let $X$ be a smooth cubic surface defined over $\CC\Semr{t}$ and let $\wht{t}: \SP \CC\Semr{t} \to \overline{M}$ be the induced morphism to the GIT compactification of smooth cubic surfaces. If after a base change the morphism can be completed to a morphism $\wht{s}: \SP \CC[[s]] \to \overline{M}^{\text{s}}$ to the locus of stable cubic surfaces, then there is a unique $R$-equivalence class on the $\CC\Semr{t}$-points of $X$.
\end{thm}
The same statement should also holds in positive characteristic provided the characteristic is large enough (probably at least $7$) such that the degree of the base change needed is divisible in the field. The GIT compactification of smooth cubic surfaces is constructed by Mumford \cite{MumfordGIT} (p. 80). The stable cubic surfaces are cubic surfaces with at worst ordinary double points as singularities. The unique strictly semistable cubic surface is given by the equation $XYZ+W^3=0$, which has $3$ $A_2$ singularities.

\textbf{Acknowledgment:} After the first version of the paper (which only deals with del Pezzo surfaces of degree $4$) was submitted, the referee found a gap in the original argument. I would like to thank him for his careful reading and many suggestions. After that, Colliot-Th\'el\`ene found another gap in the argument. I would like to thank him for his interest and many helpful discussions.

\section{$R$-equivalence and quadratic field extension}

We recall the following construction in \cite{ManinCubic}, Section 15, Proposition 15.1, and \cite{KollarCubic}, Example 3.8, Exercise 3.11.
\begin{constr}\label{WeilTrick}
Let $L/K$ be a separable quadratic field extension and $\tilde{X}_K$ a smooth projective cubic surface over $K$. Denote by $\tilde{X}_L$ the base change of $\tilde{X}_K$ to the field $L$. Consider the Weil restriction of scalars $Res_{L/K} \tilde{X}_{L}$. There is a rational dominant map (\cite{KollarCubic} Ex. 3.11, \cite{ManinCubic}, Section 15, Proposition 15.1):
\[
Res_{L/K} \tilde{X}_{L} \dashrightarrow \tilde{X}_K.
\]
When base changed to an algebraic closure or the field $L$, this can be described as mapping a pair of points to the third intersection point of the cubic surface with the line spanned by the pair (if the line does not lie in the cubic surface). In particular, the fiber (after base-changed to an algebraic closure or the field $L$) over a point is an open subset of the blow-up of the cubic surface at the point, and thus, rational.
\end{constr}
In general, the map is not surjective on $K$-rational points. However, we do have the following observation.

\begin{lem}
Given any $K$-rational point $x$ in $\tilde{X}_K$, there is one geometrically irreducible component of the fiber containing an open subset of a smooth rational surface. 
\end{lem}

\begin{proof}
This follows from the geometry of the map as discussed below.

Since we are only interested in the geometry, it suffices to consider the base change of the rational map $Res_{L/K} \tilde{X}_{L} \dashrightarrow \tilde{X}_K$ over a separably closed field. Then the rational map is the same as the rational map:
\[
\tilde{X} \times \tilde{X} \dashrightarrow \tilde{X}
\]
which sends two points $x, y$ in $\tilde{X}$ to the third intersection point of the line spanned by them provided that the line is not contained in $\tilde{X}$. Consider the graph $\Gamma \subset \tilde{X} \times \tilde{X} \times \tilde{X}$ of the above rational map.

There is a open subvariety $W \subset \Gamma$ parameterizing all the triples $(x, y, z)$ such that the line spanned by the three points does not lie in the cubic surface $X$.

Let $p_i$ be the projection of $W$ to the $i$-th factor.  Then
\begin{itemize}
\item For a point $z \in X$ which is not contained in a line, the fiber of the morphism $p_3: W \to X$ over $z$ is isomorphic to the blow-up of the surface $X$ at $z$.
\item For a point $z \in X$ which is contained in a line, the fiber of the morphism $p_3$ over $z$ is isomorphic to the complement of the lines containing $z$
\end{itemize}
Moreover, the fiber dominates the first (or the second) factor via the projection.

To see this, one can consider the projection to the first factor. Denote by $F$ the fiber of $W$ over $z$. Then in the first case the morphism $p_1: F \to \tilde{X}$ is birational and is an isomorphism over the complement of the point $z$. The preimage of each point $x \neq z$, the preimage in $F$ is a unique point $(x, y, z)$. The fiber of $F \to \tilde{X}$ over $z$ are points of the form $(z, y, z)$. They corresponds to tangent lines of $\tilde{X}$ at $z$. So the fiber over $z$ is a $\PP^1$. In the second case, the morphism is birational and is an isomorphism over the complement of the lines containing $z$. For each point $x$, not contained in the lines that contain the point $z$, there is a unique point $(x, y, z)$ in $F$ which is mapped to $x$.

The two claims together give the geometrically integral rational surface over a point . 
\end{proof}

Assume $K$ is $C_1$. Then there is at least one rational point in a birational modification of this irreducible component of the fiber(\cite{CTC1}, Proposition 2). Thus given any $K$-rational point $x$ in $\tilde{X}_K$, there is a $K$-rational point of $Res_{L/K}\tilde{X}_L$ which is mapped to $x$. This allows us to prove the following.

\begin{lem}\label{lem:reduction}
Let $K$ be a $C_1$-field of characteristic $0$ or the Laurent field $k\Semr{t}$ over an algebraically closed field $k$. And let $\tilde{X}_K$ be a smooth cubic surface defined over $K$ and $L/K$ a separable degree $2$ field extension. Finally let $\tilde{X}_L$ be the base change of $\tilde{X}_K$ to $L$. If there is a unique $R$-equivalence class on $\tilde{X}_L(L)$, then there is a unique $R$-equivalence class on $\tilde{X}_K(K)$.
\end{lem}

\begin{proof}
The previous construction \ref{WeilTrick} shows that $\tilde{X}_{K}$ is dominated by a variety (i.e. $Res_{L/K}\tilde{X}_L$) whose rational points are $R$-connected. Moreover, the fiber over any rational point in $\tilde{X}_{K}(K)$ is an open subset of a rational surface (which is isomorphic to the blow-up of the cubic surface at the rational point over the field $L$).

Over a Laurent field (or more generally over any large fields), being $R$-connected is the same as being directly $R$-connected \cite{KollarSpecialization}. Thus any two points in $\tilde{X}_{L}$, hence also any two points in $Res_{L/K}\tilde{X}_L$, are directly $R$-connected by a single rational curve. Since the fiber in $Res_{L/K}\tilde{X}_L$ over any rational point in $\tilde{X}_K$ contains a geometrically irreducible component which has an open subset isomorphic to an open subset of a smooth (geometrically) rational surface, and the field $K$ is a large field and $C_1$, the set of $K$-rational points in the open subset of $Res_{L/K}\tilde{X}_L$ where the rational map to $\tilde{X}_K$ is defined maps surjectively to the set of rational points $\tilde{X}_K(K)$. It follows that any two $K$-rational points in $\tilde{X}_K(K)$ are directly $R$-connected.

We use the following argument in the case of characteristic $0$ $C_1$-field. We first resolve the indeterminacy by $\overline{Res} \to Res_{L/K}\tilde{X}_L$. Since the set of rational points modulo $R$-equivalence is a birational invariant in characteristic $0$ (Proposition 10, \cite{CTS} p.195), there is only one $R$-equivalence class on $\overline{Res}(K)$. By previous discussion, there is a geometrically irreducible component of the fiber over any rational point on $\tilde{X}_{K}$, which is geometrically rational. A resolution of singularities of this irreducible component is a smooth projective (geometrically) rational surface over a $C_1$-field. In particular, there is a $K$-rational point on the resolution (\cite{CTC1}, Proposition 2), which is mapped to a $K$-rational point of the irreducible component. Thus the set of rational points of $\overline{Res}(K)$ maps surjectively to $\tilde{X}_K(K)$. Then one can deduce $R$-equivalence of any two $K$-rational points in $\tilde{X}_{K}$ by lifting them to $\overline{Res}(K)$.
\end{proof}
\section{Proof of Theorem \ref{thm:degree4}}

To begin the proof, first notice the following fact.

\begin{lem}\label{lem:Laurent}
Let $\tilde{X}$ be a smooth cubic surface with a line defined over a the Laurent field $K=k\Semr{t}$ over an algebraically closed field $k$ of characteristic not equal to $2$ or $5$. Then there is a sequence of degree $2$ Galois field extensions $K=K_0 \subset K_1 \subset \ldots \subset K_n$ such that the base change $\tilde{X}_n$ is rational over $K_n$.
\end{lem}
\begin{proof}
We use the following construction: projection from the line to get a conic bundle structure of $\tilde{X}$ over $\PP^1_K$ with five degenerate fibers.

First consider the case the five degenerate fibers form an irreducible cycle defined over $K$. There are two further possibilities according to the field of definitions of the $10$ lines in the degenerate fibers. These lines are defined either in a degree $5$ (separable) field extension $K'=k\Semr{s}, s^5=t$ or a degree $10$ (separable) field extension $K'=k\Semr{s}, s^{10}=t$. In the first case the $10$ lines form two irreducible cycles over $K$, each consisting of $5$ lines in distinct fibers. And we can contract one set of $5$ lines in distinct fibers of the conic fibration over $K$. Thus we get a rational surface over $K$. In the second case we can make the contraction after a (separable) degree $2$ field extension.

If there is a degree $2$ cycle of the singular fibers over $K$, we first make a degree $2$ field extension so that the two singular fibers are both defined over the field. Then after a possible degree $2$ field extension, we may assume that all the four lines are defined over the field. Then we can contract two disjoint lines and get a del Pezzo surface of degree $5$ with a rational point. Thus the surface is rational over the field.

For the case where there is just one singular fiber defined over $K$ and the other $4$ singular fibers are conjugate over $K$, we know there is a tower of two separable degree $2$ field extension $K \subset K_1 \subset K_2$ so that each of the singular fiber cycle is defined over $K_2$. Then up to making a further degree $2$ field extension, all the fibers are defined over the field. Then the base change is rational since we can contract $2$ disjoint lines and a del Pezzo surface of degree $5$ with rational points is rational. 
\end{proof}

Now we can finish the proof of the theorem.
\begin{proof}[Proof of Theorem \ref{thm:degree4}]
The Laurent field $K$ is a large field. In particular, the set of rational points is Zariski dense once we have a point \cite{Pop96}. In any case there is a general point $x$ not on a line. Blow up at the point $x$. Then we have a smooth cubic surface $\tilde{X}$ with a line corresponding to the exceptional divisor. It suffices to show that there is a unique $R$-equivalence class on $\tilde{X}$.

Projection from the line gives a conic bundle structure on $\tilde{X}$ with five degenerate fibers.

By Lemma \ref{lem:Laurent}, there is a sequence of separable degree $2$ field extensions $K=K_0 \subset K_1 \subset \ldots \subset K_n$ such that $\tilde{X}_n =\tilde{X}\times_K K_n$ is rational.

The $R$-equivalence on $\tilde{X}_n$ is trivial. And we use Lemma \ref{lem:reduction} to finish the proof.
\end{proof}

\section{Proof of Theorem \ref{thm:cubic}}
By assumption, after a degree $l$ base change $t=s^l$, there is a projective family $\mcX \to \SP \CC[[s]]$ such that
\begin{enumerate}
\item The generic fiber of the family is isomorphic to the base change of $X$ to $\SP \CC\Semr{s}$.
\item The central fiber $\mcX_0$ is a cubic surface with at worst ordinary double points as singularities.
\item The Galois group $G=\ZZ/l \ZZ$ acts on $\mcX$ and the projection to $\SP \CC[[s]]$ is $G$-equivariant.
\end{enumerate}

Any $\CC\Semr{t}$-points of $X$ induces a $G$-equivariant section of the family.

For any family of surfaces with at worst du Val singularities over $\SP \CC[[s]]$, there is a simultaneous resolution after a base change (
\cite{ArtinSimultaneous}), i.e. a diagram as the following:
\[
\begin{CD}
\tilde{\mcX}@>>>\mcX' @>>>\mcX\\
@VVV                  @ VVV			@VVV\\
\SP\CC[[s']]@=\SP\CC[[s']]@>>>\SP \CC[s]\\
\end{CD}
\] 
such that $\tilde{\mcX} \to \SP \CC[[s']]$ is a smooth proper family. In general $\tilde{\mcX}$ is only an algebraic space but in our case $\tilde{\mcX}$ is a scheme and there is a relatively amply line bundle over $\tilde{\mcX}$ (thus the family is still projective).

So we assume that there is already a simultaneous resolution $\tilde{\mcX} \to \mcX \to \SP \CC[[s]]$. However, the $G$-action in general does not lift to the simultaneous resolution.

As an example, consider the family $xy+z^2=s^2$, obtained from a degree $2$ base change of the family $xy+z^2=s$, with the natural $\ZZ/2\ZZ$ action $s \mapsto -s$. There are two simultaneous resolutions by blowing up $x=z-s=0$ or $x=z+s$=0. And the action does not extend to the resolution since it does not preserve the blow-up center. But for ordinary double points, one can lift the action of a subgroup $G_1$ whose index is a power of $2$, basically because there are exactly two different simultaneous resolutions for a ordinary double point (In the example, this is the trivial subgroup).  A precise formulation is to use Artin's result in \cite{ArtinSimultaneous}. 

To be more precise, given a family of surfaces $\mcX \to S$ over a base, consider the functor 
\[
\text{R}_{\mcX/S}(S')=\text{the set of simultaneous resolutions of}~\mcX \times_S S' ~\text{for}~S' \to S.
\]
Artin proved that this functor is represented by a quasi-separated algebraic space $R$ (here we used a notation different from the one used by Artin so that we do not confuse this with the Weil restriction introduced in previous sections). For the family
\[
xy+z^2=s^2
\]
over $\SP \CC[[s]]$, $R$ is obtained by gluing two copies of $\SP \CC[[s]]$ along the generic point via the identity map, thus is non-separated and the $\ZZ/2\ZZ$ action permutes the two closed points. 

Every family of surfaces whose generic fiber is smooth and whose central fiber has only ordinary double points as singularities has a simultaneous resolution without the need of base change if and only if locally around each singularity, it is given by
\[
xy+z^2=s^{2k}, k \in \NN.
\]
The family (locally) can be thought of as obtained from a degree $k$ base change from the $k=1$ case, and the space $R$ has a similar description as a base change.

By the universal property, the group $G$ acts on the space $R$. The action can be lifted to the resolution if and only if the closed points of the space $R$ are fixed by the action. As we have seen that this is not true in general. But by the above description, the action of a subgroup, whose index is an power of $2$, always lifts to the resolution. We may just choose the subgroup $G_1$ to be the maximal subgroup of odd order. This is clear if the ordinary double point is a fixed point of $G_1$. If it has a non-trivial orbit, then we choose one resolution for one point in the orbit and use the group action to determine resolution at all the other orbits.

The subgroup $G_1$ determines a subfield extension $\CC\Semr{t}\subset \CC\Semr{t_1}\subset \CC\Semr{s}$, where $\CC\Semr{t}\subset \CC\Semr{t_1}$ a sequence of quadratic field extensions and the Galois group of $\CC\Semr{s}/\CC\Semr{t_1}$ is isomorphic to $G_1$.

Let $X_1$ be the base change of $X$ to the field $\CC\Semr{t_1}$. 

The group $G_1$ acts on both $\tilde{\mcX}$ and $\mcX$, and the projections to $\SP \CC[[s]]$ are $G_1$-equivariant. Thus any two $\CC\Semr{t_1}$-rational points of $X_1$ induce two $G_1$-equivariant sections $\whts_1$ and $\whts_2$ of the family $\tilde{\mcX}\to \SP \CC[[s]]$, whose intersection points with the central fiber $\tilde{\mcX}_0$ are fixed points of the action of $G_1$ on $\tilde{\mcX}_0$. Denote the two points by $x$ and $y$.

By Theorem 1.4 in \cite{WAIsotrivial}, there is a $G_1$-equivariant very free curve $f: \PP^1 \to \tilde{\mcX}_0$ such that $f(0)=x, f(\infty)=y$, where the $G_1$ action on $\PP^1$ is $z \mapsto \zeta z$, and $\zeta$ is a primitive root of unity. Consider the relative Hom-scheme
\[
\text{Hom}(\PP^1 \times \SP \CC[[s]], \tilde{\mcX}, 0 \times \SP \CC[[s]] \mapsto \whts_1, \infty \times \SP \CC[[s]] \mapsto \whts_1).
\]
Here we choose the same $G_1$-action on $\PP^1$ as above. Then there is a natural $G_1$-action on the relative Hom-scheme such that the projection to the base is $G_1$-equivariant. The very free curve $f$ gives a $G_1$-fixed point of the relative Hom-scheme and the morphism to the base is smooth at this point. Then by Lemma \ref{lem:equiv-lifting} below, there is a $G_1$-equivariant section of the relative Hom-scheme, which gives a family of rational curves connecting the two $G_1$-equivariant sections. Then the two $\CC\Semr{t_1}$-rational points are $R$-equivalent. 

So after a number of quadratic field extension $\CC\Semr{t_1}/\CC\Semr{t}$, there is a unique $R$-equivalence class on the base change $X_1$ of $X$. Then the theorem follows from Lemma \ref{lem:reduction}.

\begin{lem}[\cite{WAIsotrivial}]\label{lem:equiv-lifting}
Let $X$ and $Y$ be two $\CC$-schemes with a $G$-action and $f: X \rightarrow Y$ be a finite type $G$-equivariant morphism. Let $x \in X$ be a fixed point, and $y=f(x)$ (hence also a fixed point). Assume that $f$ is smooth at $x$. Then there exists a $G$-equivariant section $s: \SP \whtOO_{y,Y}\rightarrow X$. 
\end{lem}
For the proof see Corollary 2.2, \cite{WAIsotrivial}.

\end{document}